\documentclass[12pt,reqno]{article}

\usepackage{amsmath,amssymb,amsthm,amsopn}

\usepackage{algorithm,algorithmic}

\newtheorem{definition}{Definition}
\newtheorem{lemma}{Lemma}
\newtheorem{corollary}{Corollary}
\newtheorem{proposition}{Proposition}
\newtheorem{theorem}{Theorem}

\DeclareMathOperator{\mydiv}{\: div \:}
\DeclareMathOperator{\mypos}{pos}

\begin{document}

\begin{center}
\vskip 1cm{\LARGE\bf The Number System of the Permutations Generated by Cyclic Shift}
\vskip 1cm
\large
St\'{e}phane Legendre\\
Team of Mathematical Eco-Evolution\\
Ecole Normale Sup\'{e}rieure\\
75005 Paris \\
France \\
legendre@ens.fr\\
\end{center}

\vskip .2 in

\begin{abstract}
A number system coding for the permutations generated by cyclic shift is described. The system allows to find the rank of a permutation given how it has been generated, and to determine a permutation given its rank. It defines a code describing the symmetry properties of the set of permutations generated by cyclic shift. This code is conjectured to be a combinatorial Gray code listing the set of permutations: this corresponds to an Hamiltonian path of minimal weight in an appropriate regular digraph.
\end{abstract}

\section{Introduction}\label{introduction}

Since the work of Laisant in 1888 \cite{laisant_1888} -- and even since Fischer and Krause in 1812 \cite{fischer&krause_1812} according to Hall and Knuth \cite{marshall&knuth_1965} --, it is known that the factorial number system codes for the permutations generated in lexicographic order. More precisely, when the set of all permutations on $n$ symbols is ordered by lexicographic order, the rank of a permutation written in the factorial system provides a code determining the permutation. The code specifies which interchanges of the symbols according to lexicographic order have to be performed to generate the permutation. Conversely, the rank of a permutation can be computed from its code. This coding has been rediscovered several times since (e.g., Lehmer \cite{lehmer_1960}).

In this study, we describe a number system on the finite ring $\mathbb{Z}_{n!}$ coding for the permutations generated by cyclic shift. When the set  $\mathcal{S}_{n}$ of permutations is ordered according to generation by cyclic shift, the rank of a permutation written in this number system entirely specifies how the permutation has been generated. Conversely, the rank can be computed from the code. This number system is a special case of a large class of methods presented by Knuth \cite{knuth_2005} for generating $\mathcal{S}_{n}$.

We shall describe properties of $\mathcal{S}_{n}$ generated by cyclic shift:
\begin{enumerate}
\item A decomposition into $k$-orbits;
\item The symmetries;
\item An infinite family of regular digraphs associated with $\{\mathcal{S}_{n};n \ge 1\}$;
\item A conjectured combinatorial Gray code generating the permutations on $n$ symbols. The adjacency rule associated with this code is that the last symbols of each permutation match the first symbols of the next optimally.
\end{enumerate}

\section{Number system}\label{the_number_system}

For any positive integer $a$, the ring $(\mathbb{Z}/a\mathbb{Z},+, \times)$ of integers modulo $a$ is denoted $\mathbb{Z}_{a}$. The set $\mathbb{Z}_{a}$ is identified with a subset of the set $\mathbb{N}$ of natural integers.

\begin{proposition}\label{number_system}
For $n \ge 2$, any element $\alpha \in \mathbb{Z}_{n!}$ can be uniquely represented as
\begin{displaymath}
\alpha = \sum_{i=0}^{n-2} \alpha_{i}\varpi_{n,i}, \quad \alpha_{i} \in \mathbb{Z}_{n-i},
\end{displaymath}
with the base elements 
\begin{displaymath}
\varpi_{n,0}=1, \qquad \varpi_{n,i}=n(n-1) \cdots (n-i+1), \qquad i=1, \ldots, n-2.
\end{displaymath}
\end{proposition}
The $\alpha_{i}$'s are the \textit{digits} of $\alpha$ in this number system, which we call the $\varpi$-\textit{system}. Any element of $\mathbb{Z}_{n!}$ can be written uniquely
\begin{displaymath}
\alpha = {\alpha_{n-2} \cdots \alpha_{1}\alpha_{0}}_{\varpi}.
\end{displaymath} 
Unless $\alpha_{n-2}=1$, the rightmost digits are set to 0, so that the sum always involves $n-1$ elements, indexed $0, \ldots , n-2$.

For example, in $\mathbb{Z}_{5!}$ the base is $\{\varpi_{5,0}=1,\varpi_{5,1}=5,\varpi_{5,2}=20,\varpi_{5,3}=60\}$. The element 84 writes
\begin{displaymath}
84 = 1 \times 60 + 1 \times 20 + 0 \times 5 + 4 \times 1=1104_{\varpi},
\end{displaymath}
and the element 35 writes
\begin{displaymath}
35 = 0 \times 60 + 1 \times 20 + 3 \times 5 + 0 \times 1=0130_{\varpi}.
\end{displaymath}
\begin{proof}
For simplicity, we denote $\varpi_{i}=\varpi_{n,i}$. For $n=2$, there is a single base element, $\varpi_{0}=1$, and the result clearly holds.  For $n \ge 3$, and $\alpha \in \mathbb{Z}_{n!}$, we set
\begin{displaymath}
\alpha^{(0)} = \alpha,
\end{displaymath}
\begin{displaymath}
\alpha_{i} = \alpha^{(i)} \bmod (n-i), \quad \alpha^{(i+1)}=\alpha^{(i)} \mydiv (n-i), \quad i=0, \ldots ,n-2,
\end{displaymath}
where $\mydiv$ denotes the integer division. These relations imply 
\begin{displaymath}
\alpha^{(i)} = (n-i)\alpha^{(i+1)} + \alpha_{i}, \quad i=0, \ldots ,n-2.
\end{displaymath}
We multiply by $\varpi_{i}$ on both sides. For $i=0, \ldots ,n-3$, we use the identity $\varpi_{i+1}=(n-i)\varpi_{i}$, and for $i=n-2$, we use the identity $2\varpi_{n-2}=0$ in $\mathbb{Z}_{2}$, to get 
\begin{displaymath}
\begin{array}{rcl}
\varpi_{i}\alpha^{(i)} - \varpi_{i+1}\alpha^{(i+1)} &=& \alpha_{i}\varpi_{i}, \qquad i=0, \ldots ,n-3,\\
 \varpi_{n-2}\alpha^{(n-2)} &=& \alpha_{n-2}\varpi_{n-2}.
\end{array}
\end{displaymath}
Adding these relations together, and accounting for telescoping cancellation on the left side,
\begin{displaymath}
\varpi_{0}\alpha^{(0)} = \alpha_{n-2}\varpi_{n-2} + \cdots + \alpha_{0}\varpi_{0}.
\end{displaymath}
We obtain the representation
\begin{displaymath}
\alpha = \alpha_{n-2}\varpi_{n-2} + \cdots + \alpha_{0}\varpi_{0}.
\end{displaymath}
By construction, $\alpha_{i} \in  \mathbb{Z}_{n-i}$ for $i=0, \ldots ,n-2$. The digits $\alpha_{i}$ are uniquely determined, so that the representation is unique.
\end{proof}

Arithmetics can be performed in the ring $(\mathbb{Z}_{n!},+,\times)$ endowed with the $\varpi$-system. The computation of the sum and product works in the usual way of positional number systems, using the ring structure of $\mathbb{Z}_{n-i}$ for the operations on the digits of the operands. There is no carry to propagate after the rightmost digit.

\begin{lemma}\label{the_recurrence_relation}
The base elements verify
\begin{equation}\label{recurrence_relation} 
\varpi_{n,i+k} = \varpi_{n-k,i}\varpi_{n,k},
\end{equation}
\begin{equation}\label{sum_k} 
\sum_{i=0}^{k-1}(n-i-1)\varpi_{n,i} = \varpi_{n,k}-1, \quad k \in \{1, \ldots  n-2\},
\end{equation}
\begin{equation}\label{sum_moins1} 
\sum_{i=0}^{n-2}(n-i-1)\varpi_{n,i} = -1.
\end{equation}
\end{lemma}
\begin{proof}
The verification of the first relation is straightforward. For the two other relations, let 
\begin{displaymath}
\xi = \sum_{i=0}^{k-1}\alpha_{i}\varpi_{n,i}, \quad \alpha_{i}=n-i-1 \in \mathbb{Z}_{n-i}.
\end{displaymath}
In $\mathbb{Z}_{n-i}$, $\alpha_{i}+1=0$. Therefore, when computing $\xi+1$, the carry propagates from $\alpha_{0}$ up to $\alpha_{k-1}$, and the $\alpha_{i}$'s are set to 0. If $k \le n-2$, the digit $\alpha_{k}=0$ gets the carry and is replaced by 1. In this case, $\xi+1=\varpi_{n,k}$. If $k=n-1$, there is no carry to propagate after the rightmost digit, and all digits of $\xi+1$ are set to 0. In this case, $\xi+1=0$.
\end{proof}

\begin{corollary}\label{alpha_plus_alphaprime}
For $\alpha,\alpha '\in \mathbb{Z}_{n!}$, with digits $\alpha_{i},\alpha '_{i}\in \mathbb{Z}_{n-i}$,
\begin{displaymath}
\alpha + \alpha ' = -1 \iff \alpha_{i} + \alpha '_{i}= -1, \quad i=0, \ldots, n-2.
\end{displaymath} 
\end{corollary}
\begin{proof}
We write 
\begin{displaymath}
\alpha + \alpha ' = \sum_{i=0}^{n-2}(\alpha_{i}+\alpha '_{i})\varpi_{n,i}.
\end{displaymath}
In $\mathbb{Z}_{n-i}$, $\alpha_{i}+\alpha '_{i}=-1$ if and only if $\alpha_{i}+\alpha '_{i}=n-i-1$. By uniqueness of the decomposition in the $\varpi$-system, the result follows from  
(\ref{sum_moins1}).
\end{proof}

It can be noted that (\ref{sum_moins1}) leads in $\mathbb{N}$ to the identity 
\begin{displaymath}
\sum_{i=0}^{n-2}(n-i-1)\frac{n!}{(n-i)!} = n!-1,
\end{displaymath}
which is related to the identity
\begin{equation}\label{factorial_system}
\sum_{i=1}^{n-1}i \cdot i! = n!-1,
\end{equation}
associated with the factorial number system. Identities (\ref{sum_k}) and (\ref{sum_moins1}) are instances of general identities of mixed radix number systems.

\section{Code}\label{the_code}
The set of permutations on $n$ symbols $x_{1},\ldots,x_{n}$ is denoted $\mathcal{S}_n$. From a permutation $q$ on the $n-1$ symbols $x_{1},\ldots,x_{n-1}$, $n$ permutations on $n$ symbols are generated by inserting $x_{n}$ to the right and cyclically permuting the symbols. The insertion of $x_{n}$ to the right defines an injection
\begin{displaymath}
\begin{array}{ccc}
\mathcal{S}_{n-1} &\overset{\iota}{\longrightarrow}& \mathcal{S}_{n}\\
q = (a_{1} \cdots a_{n-1}) &\longmapsto& (a_{1} \cdots a_{n-1}x_{n}) = \tilde{q}.
\end{array}
\end{displaymath}
We define the \textit{cyclic shift} $S: \mathcal{S}_{n-1} \rightarrow \mathcal{S}_{n}$ by $S=C \circ \iota$, where $C: \mathcal{S}_{n} \rightarrow \mathcal{S}_{n}$ is the circular permutation, so that
\begin{eqnarray}
\nonumber S^{0}q   &= &(a_{1}a_{2}\cdots a_{n-1}x_{n})=C^{0}\tilde{q} = \tilde{q},\\
\nonumber S^{1}q   &= &(a_{2}\cdots a_{n-1}x_{n}a_{1})=C^{1}\tilde{q},\\
\nonumber &\vdots\\
\nonumber S^{n-1}q &= &(x_{n}a_{1}a_{2}\cdots a_{n-1})=C^{n-1}\tilde{q}.
\end{eqnarray}
The set $\mathcal{O}(q)=\{S^{0}q, \ldots,S^{n-1}q\}$ is the \textit{orbit} of $q$. As $S^{i}=S^{j}$ is equivalent to $i = j \bmod n$, the exponents of the cyclic shift are elements of $\mathbb{Z}_{n}$.

\begin{lemma}\label{orbits}
The set of permutations $\mathcal{S}_{n}$ is the disjoint union of the orbits $\mathcal{O}(q)$ for $q \in \mathcal{S}_{n-1}$.
\end{lemma}
\begin{proof} 
If $q,r \in \mathcal{S}_{n-1}$, their orbits are disjoint subsets of $\mathcal{S}_{n}$. Indeed, if $S^{i}q=S^{j}r$ there exists $k \in \mathbb{Z}_{n}$ such that $S^{k}q=S^{0}r=\tilde{r}$. The only possibility is $k=0$, implying $S^{0}q=\tilde{q}=\tilde{r}$, and $q=r$. There are $(n-1)!$ disjoint orbits, each of size $n$, so that they span $\mathcal{S}_{n}$.
\end{proof}

According to Lemma \ref{orbits}, the set $\mathcal{S}_{n}$ can be generated by cyclic shift. The generation by cyclic shift defines an order on the set of permutations, $\mathcal{S}_{n}= \{p_{0}, \ldots ,p_{n!-1}\}$, indexed from 0 (a cyclic order in fact). For this order, the rank $\alpha$ of a permutation $p_{\alpha} \in \mathcal{S}_{n}$ is an element of $\mathbb{Z}_{n!}$. 

The generation by cyclic shift of $p \in \mathcal{S}_{n}$ from $(1) \in \mathcal{S}_{1}$ can be schematized:
\begin{equation}\label{schema}
\left\{
\begin{array}{c}
p^{(1)}=(1) \overset{\alpha_{n-2}}{\longrightarrow} p^{(2)} \longrightarrow \cdots  \overset{\alpha_{2}}{\longrightarrow} p^{(n-2)}\overset{\alpha_{1}}{\longrightarrow} p^{(n-1)}\overset{\alpha_{0}}{\longrightarrow}  p^{(n)}=p,\\
\\
p^{(n-i)} = S_{n-i}^{\alpha_{i}}p^{(n-i-1)},
\end{array}
\right\}
\end{equation}
where $p^{(n-i)} \in \mathcal{S}_{n-i}$ is generated from $p^{(n-i-1)}\in \mathcal{S}_{n-i-1}$ by the cyclic shift 
\begin{displaymath}
S_{n-i}:\mathcal{S}_{n-i-1} \longrightarrow \mathcal{S}_{n-i}
\end{displaymath}
with the exponent $\alpha_{i} \in \mathbb{Z}_{n-i}$. 

\begin{definition}\label{code}
The sequence of exponents associated with successive cyclic shifts leading from $(1) \in \mathcal{S}_{1}$ to $p \in \mathcal{S}_{n}$ is the code of $p$ in the $\varpi$-system:
\begin{displaymath}
\alpha = {\alpha_{n-2} \cdots \alpha_{0}}_{\varpi} \in \mathbb{Z}_{n!}.
\end{displaymath}
\end{definition}

\begin{theorem}\label{permutation_code}
The rank of a permutation on $n$ symbols generated by cyclic shift is given by its code. A permutation on $n$ symbols generated by cyclic shift is determined by writing its rank in the $\varpi$-system.
\end{theorem}
\noindent For example, the permutation $p_{84}=(51324) \in \mathcal{S}_{5}$ is generated:
\begin{displaymath}
(1) \xrightarrow{\alpha_{3}=1} (21) \xrightarrow{\alpha_{2}=1} (132) \xrightarrow{\alpha_{1}=0} (1324) \xrightarrow{\alpha_{0}=4} (51324).
\end{displaymath}
Its code is $1104_{\varpi}=84$.
\begin{proof}
We use induction on $n$. For $n=2$, in $\mathcal{S}_{2}=\{(12),(21)\}$, the rank of the permutation $(12)$ is $0 = 0_{\varpi}$, and the rank of the permutation $(21)$ is $1 = 1_{\varpi}$. For $n>2$, let $p=p_{\alpha} \in \mathcal{S}_{n}$ of rank $\alpha$, generated by cyclic shift from $q=q_{\beta} \in \mathcal{S}_{n-1}$ of rank $\beta$. Then $p = S_{n}^{\alpha_{0}}q$ for some $\alpha_{0} \in \mathbb{Z}_{n}$, $\alpha_{0}$ being the rank of $p$ within the orbit of $q$. As the orbits contain $n$ elements and as $\beta$ is the rank of $q$ in $\mathcal{S}_{n-1}$, the rank of $p$ in $\mathcal{S}_{n}$ is 
\begin{displaymath}
\alpha = \beta n + \alpha_{0} = \beta\varpi_{n,1} + \alpha_{0}\varpi_{n,0}.
\end{displaymath}
By induction hypothesis, the rank $\beta$ of $q$ is given by the code
\begin{displaymath}
\beta = \sum_{i=0}^{n-3} \beta_{i}\varpi_{n-1,i}, \quad \beta_{i} \in \mathbb{Z}_{n-1-i}.
\end{displaymath}
For $k=1$, Eq. (\ref{recurrence_relation}) gives
\begin{displaymath}
\varpi_{n,i+1} = \varpi_{n-1,i}\varpi_{n,1},
\end{displaymath}
so that
\begin{displaymath}
\beta\varpi_{n,1} = \sum_{i=0}^{n-3} \beta_{i}\varpi_{n-1,i}\varpi_{n,1} = \sum_{i=0}^{n-3} \beta_{i}\varpi_{n,i+1} = \sum_{i=1}^{n-2} \beta_{i-1}\varpi_{n,i}.
\end{displaymath} 
Let $\alpha_{i}=\beta_{i-1}$ for $i=1, \ldots ,n-2$. As $\beta_{i} \in \mathbb{Z}_{n-1-i}$, $\alpha_{i} \in \mathbb{Z}_{n-i}$. We obtain that
\begin{displaymath}
\alpha = \beta\varpi_{n,1} + \alpha_{0}\varpi_{n,0} = \sum_{i=0}^{n-2} \alpha_{i}\varpi_{n,i}, \quad \alpha_{i} \in \mathbb{Z}_{n-i},
\end{displaymath} 
is the code of $p_{\alpha}$. Conversely, let $p_{\alpha} \in \mathcal{S}_{n}$. We write the rank $\alpha$ in the $\varpi$-system, $\alpha = {\alpha_{n-2} \cdots \alpha_{0}}_{\varpi}$, and  use scheme (\ref{schema}) -- from right to left -- with the exponents $\alpha_{0}, \ldots ,\alpha_{n-2}$ to determine $p_{\alpha}$.
\end{proof}

We end the section by a package of algorithms performing the correspondance rank $\leftrightarrow$ permutation of Theorem \ref{permutation_code}. Permutations are represented by strings indexed from 1. Algorithm C in Knuth \cite{knuth_2005} generates $\mathcal{S}_{n}$ by cyclic shift in a simple version of the scheme described in this section.

\noindent \hrulefill
\begin{algorithmic}
\footnotesize
\STATE $\textsc{Int2Num}(n,\alpha)$ \{ conversion from integer to $\varpi$-system \}
\FOR{$i \gets 0$ to $n-2$}
\STATE $A[i] \gets \alpha \bmod (n-i)$
\STATE $\alpha \gets \alpha \mydiv (n-i)$
\ENDFOR
\RETURN $A$
\end{algorithmic}
\noindent \hrulefill
\begin{algorithmic}
\footnotesize
\STATE $\textsc{Num2Int}(n,A)$ \{ conversion from $\varpi$-system to integer \}
\STATE $\alpha \gets 0$
\STATE $base \gets 1$
\FOR{$i \gets 0$ to $n-2$}
\STATE $\alpha \gets \alpha + A[i]*base$
\STATE $base \gets base*(n-i)$
\ENDFOR
\RETURN $\alpha$
\end{algorithmic}
\noindent \hrulefill
\begin{algorithmic}
\footnotesize
\STATE $\textsc{Circ}(m,k,p)$ \{ Circular permutation of exponent $k$ on $m$ symbols \}
\FOR{$i \gets 1$ to $k$}
\STATE $c \gets p[1]$
\FOR{$j \gets 2$ to $m$}
\STATE $p[j-1] \gets p[j]$
\ENDFOR
\STATE $p[m] \gets c$
\ENDFOR
\RETURN $p$
\end{algorithmic}
\noindent \hrulefill
\begin{algorithmic}
\footnotesize
\STATE $\textsc{Pos}(m,p)$ \{ Position of $x_{m}$ in a permutation $p$ on $m$ symbols \}
\FOR{$j \gets 1$ to $m$}
\IF{$p[j] = x_{m}$}
\RETURN $m-j$
\ENDIF
\ENDFOR
\end{algorithmic}
\noindent \hrulefill
\begin{algorithmic}
\footnotesize
\STATE $\textsc{Perm2Rank}(n,p)$ \{ Find the rank of a given permutation $p$ \}
\FOR{$i \gets 0$ to $n-2$}
\STATE $m \gets n-i$
\STATE $A[i] \gets \textsc{Pos}(m,p)$
\STATE $p \gets \textsc{Circ}(m,m-A[i],p)$
\ENDFOR
\STATE $\alpha \gets \textsc{Num2Int}(n,A)$
\RETURN $A$
\end{algorithmic}
\noindent \hrulefill
\begin{algorithmic}
\footnotesize
\STATE $\textsc{Rank2Perm}(n,\alpha)$ \{ Determine a permutation given its rank $\alpha$ \}
\STATE $A \gets \textsc{Int2Num}(n,\alpha)$
\STATE $p \gets x_{1}$
\FOR{$i \gets n-2$ downto $0$}
\STATE $m \gets n-i$
\STATE $p \gets \textsc{Circ}(m,A[i],p+x_{m})$
\ENDFOR
\RETURN $p$
\end{algorithmic}
\noindent \hrulefill
\begin{algorithmic}
\footnotesize
\STATE $\textsc{SetPerm}(n)$ \{ Generation of the permutations on $n$ symbols \}
\FOR{$\alpha \gets 0$ to $n!-1$}
\STATE $p \gets \textsc{Rank2Perm}(n,\alpha)$
\ENDFOR\\
\end{algorithmic}
\noindent \hrulefill\\

In the sequel, we assume that the set of permutations $\mathcal{S}_{n}$ is ordered according to generation by cyclic shift.

\begin{table}[ht]
\footnotesize
\begin{center}
\begin{tabular}{| c c c c c |}
\hline
$\alpha$ &$p_{\alpha}$ &$\alpha_{2}$ &$\alpha_{1}$ &$\alpha_{0}$\\
\hline
0	&1234	&0	&0 &0\\
1	&2341	&0	&0	&1\\
2	&3412	&0	&0	&2\\
3	&4123	&0	&0	&3\\
4	&2314	&0	&1	&0\\
5	&3142	&0	&1	&1\\
6	&1423	&0	&1	&2\\
7	&4231	&0	&1	&3\\
8	&3124	&0	&2	&0\\
9	&1243	&0	&2	&1\\
10	&2431	&0	&2	&2\\
11	&4312	&0	&2	&3\\
12	&2134	&1	&0	&0\\
13	&1342	&1	&0	&1\\
14	&3421	&1	&0	&2\\
15	&4213	&1	&0	&3\\
16	&1324	&1	&1	&0\\
17	&3241	&1	&1	&1\\
18	&2413	&1	&1	&2\\
19	&4132	&1	&1	&3\\
20	&3214	&1	&2	&0\\
21	&2143	&1	&2	&1\\
22	&1432	&1	&2	&2\\
23	&4321	&1	&2	&3\\
\hline
\end{tabular}
\end{center}
\caption{The codes of the permutations of $\{1,2,3,4\}$ generated by cyclic shift.}
\label{table_S4}
\end{table}

\section{$k$-orbits}\label{the_k-orbits}

In this section, structural properties of $\mathcal{S}_{n}$ are described using the $\varpi$-system.

\begin{proposition}\label{beta_gamma_decomposition}
Let $k \in \{0, \ldots ,n-2\}$ and $p_{\alpha} \in \mathcal{S}_{n}$ with code $\alpha \in \mathbb{Z}_{n!}$. There exists a permutation $q_{\beta} \in S_{n-k}$ with code $\beta \in \mathbb{Z}_{(n-k)!}$ such that
\begin{equation}\label{beta_gamma}
\alpha = \beta\varpi_{n,k} + \gamma, \quad \gamma \in \{0, \ldots ,\varpi_{n,k} - 1\}.
\end{equation}
The code $\beta$ is made of the $n-k-1$ leftmost digits of $\alpha$, and $\gamma$ is made of the $k$ rightmost digits of $\alpha$. 
\end{proposition}
\begin{proof}
We have the decomposition
\begin{displaymath}
\alpha = {\alpha_{n-2} \cdots \alpha_{0}}_{\varpi} = {\alpha_{n-2} \cdots \alpha_{k}0 \cdots 0}_{\varpi} + {0 \cdots 0 \alpha_{k-1} \cdots \alpha_{0}}_{\varpi} = \tilde{\alpha} + \gamma.
\end{displaymath}
Let $\beta_{i}=\alpha_{i+k}$ for $i=0, \ldots ,n-k-2$, so that the $\beta$'s are the $n-k-1$ leftmost digits of $\alpha$. As $\alpha_{i} \in \mathbb{Z}_{n-i}$, $\beta_{i}=\alpha_{i+k} \in \mathbb{Z}_{n-k-i}$. Hence 
\begin{displaymath}
\beta = \sum_{i=0}^{n-k-2} \beta_{i}\varpi_{n-k,i}, \quad \beta_{i} \in \mathbb{Z}_{n-k-i},
\end{displaymath} 
is an element of $\mathbb{Z}_{(n-k)!}$ which is the code of a permutation $q_{\beta} \in \mathcal{S}_{n-k}$. 
Using relation (\ref{recurrence_relation}), we obtain
\begin{displaymath}
\tilde{\alpha} = \sum_{i=k}^{n-2} \alpha_{i}\varpi_{n,i} = \sum_{i=0}^{n-k-2} \alpha_{i+k}\varpi_{n,i+k} = \sum_{i=0}^{n-k-2} \alpha_{i+k}\varpi_{n-k,i}\varpi_{n,k} = \left( \sum_{i=0}^{n-k-2} \beta_{i}\varpi_{n-k,i}\right)\varpi_{n,k}.
\end{displaymath} 
The term
\begin{displaymath}
\gamma = \sum_{i=0}^{k-1}\alpha_{i}\varpi_{n,i}
\end{displaymath}
is made of the $k$ rightmost digits of $\alpha$. It is an element of $\mathbb{Z}_{n-k+1}\times \cdots \times \mathbb{Z}_{n}$ ranging from 0 to $\sum_{i=0}^{k-1}(n-i-1)\varpi_{n,i}$, which equals $\varpi_{n,k}-1$ by (\ref{sum_k}). We obtain
\begin{displaymath}
\alpha = \tilde{\alpha} + \gamma = \beta\varpi_{n,k} + \gamma.
\end{displaymath} 
\end{proof}

\begin{definition}\label{def_k_orbits}
For $k \in \{0, \ldots ,n-2\}$, and $q_{\beta} \in \mathcal{S}_{n-k}$, the $k$-\textit{orbit} of $q_{\beta}$ in $\mathcal{S}_{n}$ is the subset
\begin{displaymath}
\mathcal{O}_{n,k}(q_{\beta})=\{p_{\alpha} \in S_{n}; \quad \alpha = \beta\varpi_{n,k} + \gamma, \quad \gamma=0, \dots ,\varpi_{n,k}-1\}.
\end{displaymath} 
\end{definition}

For $k=0$, the 0-orbit of $q \in \mathcal{S}_{n}$ is $\{q\}$. Indeed, for $k=0$, $\varpi_{n,0}=1$, $\gamma=0$, and $q=p_{\alpha}$. For $k \ge 1$, a $k$-orbit $\mathcal{O}_{n,k}(q)$ can be described as the subset of $\mathcal{S}_{n}$ generated from $q \in \mathcal{S}_{n-k}$ by $k$ successive cyclic shifts. Indeed, by Proposition \ref{beta_gamma_decomposition}, the code of $p_{\alpha} \in \mathcal{O}_{n,k}(q_{\beta})$ is obtained by appending ${\alpha_{k-1} \cdots \alpha_{0}}$ to the code ${\beta_{n-k-2} \cdots \beta_{0}}_{\varpi}$ of $q_{\beta}$. By scheme (\ref{schema}), the digits $\alpha_{k-1}, \ldots, \alpha_{0}$ describe the generation of $p_{\alpha}$ from $q_{\beta}$. In particular, for $k=1$, the 1-orbit $\mathcal{O}_{n,1}(q)$ of $q \in \mathcal{S}_{n-1}$ is the orbit $\mathcal{O}(q)$. We may further define the $(n-1)$-orbit $\mathcal{O}_{n,n-1}(q)$ as the whole set $\mathcal{S}_{n}$, with $q=(1) \in \mathcal{S}_{1}$.

We have the following generalization of Lemma \ref{orbits}:

\begin{proposition}\label{k-orbits}
For $k \in \{0, \ldots ,n-2\}$, the set of permutations $\mathcal{S}_{n}$ is the disjoint union of the $k$-orbits $\mathcal{O}_{n,k}(q)$ for $q \in \mathcal{S}_{n-k}$.
\end{proposition}
\begin{proof}
The $k$-orbits are disjoint by uniqueness of the decomposition (\ref{beta_gamma}). They are in number $(n-k)!$ and contain $\varpi_{n,k}$ elements each. As $(n-k)!\varpi_{n,k}=n!$ in $\mathbb{N}$, the $k$-orbits span $\mathcal{S}_{n}$.
\end{proof}
In decomposition (\ref{beta_gamma}), $\beta$ specifies to which $k$-orbit $p_{\alpha}$ belongs and $\gamma$ specifies the rank of $p_{\alpha}$ within the $k$-orbit. The first element of the $k$-orbit has rank $\alpha^{first} = \beta \varpi_{n,k}$ (i.e., $\gamma=0$). The last element has rank $\alpha^{last} = \beta \varpi_{n,k} + \varpi_{n,k}-1$ (i.e., $\gamma=\varpi_{n,k}-1$). The element next to the last has rank $\alpha^{last}+1 = \beta\varpi_{n,k}+\varpi_{n,k} = (\beta+1)\varpi_{n,k}$. It is the first element of the next $k$-orbit $\mathcal{O}_{n,k}(q_{\beta+1})$, where $q_{\beta+1}$ is the element next to $q_{\beta}$ in $\mathcal{S}_{n-k}$.
\begin{proposition}\label{rank_k_orbit}
For $k \in \{0, \ldots ,n-2\}$, the digit $\alpha_{k}$ of the code of $p_{\alpha}$ is the rank of the $k$-orbit within the $(k+1)$-orbit containing $p_{\alpha}$.
\end{proposition}
For example, Table \ref{table_S4} shows that $\mathcal{S}_{4}$ contains two 2-orbits within the 3-orbit $\mathcal{S}_{4}$. The ranks $0,1$ of these 2-orbits in $\mathcal{S}_{4}$ are specified by the digit $\alpha_{2}$.
\begin{proof}
The number of $k$-orbits within a $(k+1)$-orbit is $n-k$ (indeed, $(n-k)!/(n-(k+1))! = n-k$). When performing $\beta \rightarrow \beta+1$, the digit $\beta_{0}=\alpha_{k}$ ranges from 0 to $n-k-1$ in $\mathbb{Z}_{n-k}$. It specifies the rank of the $k$-orbit within the $(k+1)$-orbit.
\end{proof}

\begin{lemma}\label{last_element_k_orbit}
Let $p_{\alpha} \in \mathcal{S}_{n}$. There exists a largest $k \in \{0, \ldots, n-2\}$ and $q_{\beta} \in \mathcal{S}_{n-k}$ such that $p_{\alpha}$ is the last element of the $k$-orbit $\mathcal{O}_{n,k}(q_{\beta})$, and not the last element of the $(k+1)$-orbit containing this $k$-orbit.
\end{lemma}
\begin{proof}
If $p_{\alpha}$ is not the last element of the 1-orbit it belongs to, it is the last element of the 0-orbit $\{p_{\alpha}\}$. In this trivial case, $k=0$ and $p_{\alpha}=q_{\beta}$. Otherwise the last digit of $p_{\alpha}$ is $\alpha_{0}=n-1$. There exists a largest $k \ge 1$ such that $\alpha_{i}=n-i-1$ for $i=0, \ldots, k-1$, and $\alpha_{k} \ne n-i-1$. This means that $p_{\alpha}$ is the last element of nested $j$-orbits, $j=1, \ldots, k$.
\end{proof}

\section{Symmetries}\label{the_symmetries}

For compatibility with the cyclic shift, we adopt the convention that the positions of the symbols in a permutation are computed from the right and are considered as elements of $\mathbb{Z}_{n}$ (the position of the last symbol is 0 and the position of the first symbol is $n-1$).

According to scheme (\ref{schema}), symbol $x_{n-i}$ ($i \ge 2$) appears at step $n-i$ with the digit $\alpha_{i}$ as exponent of the cyclic shift. Its position in the generated permutation $p^{(n-i)}$ is therefore
\begin{displaymath}
\mypos_{n-i}(x_{n-i},p^{(n-i)}) = \alpha_{i}.
\end{displaymath} 
In particular,
\begin{displaymath}
\mypos_{n}(x_{n},p^{(n)}) = \alpha_{0}.
\end{displaymath} 

For a permutation $p=(a_{1}a_{2} \cdots a_{n-1}a_{n}) \in \mathcal{S}_{n}$, we introduce the \textit{mirror image} of $p$, $\overline{p}=(a_{n}a_{n-1} \cdots a_{2}a_{1})$.

\begin{proposition}\label{mirror_image}
The permutations $p_{\alpha}$ and $p_{\alpha '}$ are the mirror image of one another if and only if 
\begin{displaymath}
\alpha + \alpha ' = -1.
\end{displaymath}
\end{proposition}
For example, in $\mathbb{Z}_{5!}$ we have $84+35 = -1$, and in $\mathcal{S}_{5}$, $p_{84}=(51324)$ is the mirror image of $p_{35}=(42315)$.

\begin{proof}
The proof is by induction on $n$. For $n=2$, $p_{0}=(12)$, $p_{1}=(21)$, and $0+1=1=-1$ in $\mathbb{Z}_{2}$. Let $n > 2$. By Proposition \ref{beta_gamma},
\begin{displaymath}
\alpha = \beta\varpi_{n,1}+\alpha_{0}\varpi_{n,0},\quad \alpha ' = \beta '\varpi_{n,1}+\alpha '_{0}\varpi_{n,0}, \quad q_{\beta},q_{\beta '} \in  \mathcal{S}_{n-1}, \quad \alpha_{0},\alpha '_{0} \in \mathbb{Z}_{n}.
\end{displaymath}
By Corollary \ref{alpha_plus_alphaprime}, the condition $\alpha+\alpha ' = -1$ is equivalent to $\beta+\beta ' = -1$  and $\alpha_{0}+\alpha '_{0} = -1$. By induction hypothesis, $q_{\beta}$ is the mirror image of $q_{\beta '}$ in $\mathcal{S}_{n-1}$ if and only if $\beta +\beta ' = -1$. The condition $\alpha +\alpha '_{0} = -1$ is equivalent to $\alpha '_{0} = n-1-\alpha_{0}$, i.e., the ranks of $\alpha_{0}$ and $\alpha '_{0}$ are symmetrical in $\mathbb{Z}_{n}$. As these ranks are the positions of symbol $x_{n}$ when $p_{\alpha}$ and $p_{\alpha '}$ are generated by cyclic shift from $q_{\beta}$ and $q_{\beta '}$ respectively, we obtain the result.
\end{proof}

\begin{corollary}\label{palindrome}
The word constructed by concatenating the symbols of the permutations generated by cyclic shift is a palindrome.
\end{corollary}

\begin{proof}
Let $p_{\alpha} \in \mathcal{S}_{n}$. The rank symmetrical to $\alpha$ in $\mathbb{Z}_{n!}$ is $(n!-1)-\alpha=-(\alpha+1)$. By Proposition \ref{mirror_image}, $p_{-(\alpha+1)}$ is the mirror image of $p_{\alpha}$.
\end{proof}

The set $\mathcal{S}_{n}$ has in fact deeper symmetries, coming from the recursive structure of the $k$-orbits. 

According to Theorem \ref{permutation_code}, the generation of $\mathcal{S}_{n}$ by cyclic shift is obtained by performing $\alpha \rightarrow \alpha + 1$ for $\alpha \in \mathbb{Z}_{n!}$, and writing $\alpha$ in the $\varpi$-system. This determines each permutation $p_{\alpha}$. As $\alpha$ runs through $\mathbb{Z}_{n!}$, $p_{\alpha}$ runs through the $k$-orbits of $\mathcal{S}_{n}$. For a fixed $k$, and by Proposition \ref{rank_k_orbit}, $p_{\alpha}$ leaves a $k$-orbit to enter the next when, in the computation of $\alpha+1$, the carry propagates up to the digit $\alpha_{k}$, incrementing the rank $\beta$ of the $k$-orbit. This occurs when $\alpha = \beta\varpi_{n,k} + \varpi_{n,k}-1$.

\begin{proposition}\label{AB_BA}
Any two successive permutations of $\mathcal{S}_{n}$ are written as
\begin{displaymath}
p_{\alpha} = \overline{A}B, \qquad
p_{\alpha+1} = BA,
\end{displaymath}
with an integer $k \in \{0, \ldots, n-2\}$ such that
\begin{displaymath}
|A|=k+1.
\end{displaymath}
\end{proposition}
For example, in $\mathcal{S}_{5}$, $p_{39}=(542\underline{31})$ and 
$p_{40}=(\underline{31}245)$, with $39={0134}_{\varpi}$ and $40={0200}_{\varpi}$.

\begin{proof}
If $p_{\alpha}$ and $p_{\alpha+1}$ are in the same 1-orbit then
\begin{displaymath}
p_{\alpha} = (a_{1}a_{2} \cdots a_{n}), \qquad p_{\alpha+1} = (a_{2}\cdots a_{n}a_{1}).
\end{displaymath}
The result holds with $A=(a_{1})$, $B=(a_{2}\cdots a_{n})$, and this corresponds to $k=0$. Otherwise, by Lemma \ref{last_element_k_orbit}, there exists a largest $k \ge 1$ such that $p_{\alpha}$ is the last element of a $k$-orbit, and not the last element of a $(k+1)$-orbit. The elements of a $k$-orbit are generated by successively inserting the symbols $x_{n-k+1}, \ldots, x_{n}$ from a permutation $q_{\beta} \in \mathcal{S}_{n-k}$. The last element is
\begin{displaymath}
(x_{n} \cdots x_{n-k+1}b_{1} \cdots b_{n-k}),
\end{displaymath}
where $q_{\beta} = (b_{1} \cdots b_{n-k})$ is a permutation of the symbols $x_{1}, \ldots, x_{n-k}$. The first element of the next $k$-orbit is
\begin{displaymath}
(c_{1} \cdots c_{n-k}x_{n-k+1} \cdots x_{n}),
\end{displaymath}
where $q_{\beta+1} = (c_{1} \cdots c_{n-k})$. As $\mathcal{S}_{n-k}$ is generated by cyclic shift, $q_{\beta+1} = C_{n-k}q_{\beta}$, with $C_{n-k}$ the circular permutation in $\mathcal{S}_{n-k}$. We can now write
\begin{displaymath}
\begin{array}{ccc}
p_{\alpha}   &= (x_{n} \cdots x_{n-k+1}b_{1}b_{2} \cdots b_{n-k}) &= \overline{A}B\\
p_{\alpha+1} &= (b_{2} \cdots b_{n-k}b_{1}x_{n-k+1} \cdots x_{n}) &= BA,
\end{array}
\end{displaymath}
where $A=(b_{1}x_{n-k+1} \cdots x_{n})$ contains $k+1$ symbols.
\end{proof}

According to the Proposition, $k+1$ symbols have to be erased to the left of $p_{\alpha}$ so that the last symbols of $p_{\alpha}$ match the first symbols of $p_{\alpha+1}$. We define the \textit{weight} $e_{n}(\alpha) \in \{1, \ldots, n-1\}$ of the transition $\alpha \rightarrow \alpha+1$ as the number of symbols of $A$ in the above decomposition of $p_{\alpha}$ and $p_{\alpha+1}$.

We define the \textit{$\varpi$-ruler sequence} as
\begin{displaymath}
E_{n}=\{e_{n}(\alpha); \quad \alpha = 0, \ldots, n!-2\}.
\end{displaymath}

\begin{proposition}\label{E_palindrome}
The $\varpi$-ruler sequence is a palindrome.
\end{proposition}
\begin{proof}
If the ranks of $\alpha$ and $\alpha '$ are symmetrical in $\mathbb{Z}_{n!}$, $\alpha+\alpha ' = -1$, and $\alpha_{i}+\alpha '_{i} = -1$ for $i=0, \ldots, n-2$ by Corollary \ref{alpha_plus_alphaprime}. By the definition of $e_{n}(\alpha)$, we want to show that $e_{n}(\alpha)=e_{n}(\alpha '-1)$. If $p_{\alpha}$ is the last element of a $k$-orbit, then $\alpha_{i} = -1$ for $i=0, \ldots, k-1$, so that $\alpha '_{i} = 0$ for $i=0, \ldots, k-1$:  $p_{\alpha '}$ is the first element of a $k$-orbit and  $p_{\alpha '-1}$ is the last element of the previous $k$-orbit. Hence $e_{n}(\alpha)=e_{n}(\alpha '-1)=k+1$. If $p_{\alpha}$ is not the last element of a $k$-orbit, then $\alpha_{0} \ne -1$, $\alpha '_{0} \ne 0$, $p_{\alpha '}$ is not the first element of a 1-orbit. In this case $e_{n}(\alpha)=e_{n}(\alpha '-1)=1$.
\end{proof}

\begin{proposition}\label{ruler_sequence}
The number of terms of the $\varpi$-ruler sequence such that $e_{n}(\alpha)=k$ is 
\begin{displaymath}
(n-k)(n-k)!. 
\end{displaymath}
The sum of its $n!-1$ terms is 
\begin{displaymath}
W_{n}=1!+ 2!+ \ldots +n!-n.
\end{displaymath}
\end{proposition}
\begin{proof} We have $e_{n}(\alpha)=k \ge 1$ if and only if $p_{\alpha}$ is the last element of a $(k-1)$-orbit, and not the last element of a $k$-orbit. The number of $(k-1)$-orbits within a $k$-orbit is $n-k+1$ (see Proposition \ref{rank_k_orbit}). We exclude the last $(k-1)$-orbit, giving $n-k$ possibilities. The number of $k$-orbits is $(n-k)!$ so that there are $(n-k)(n-k)!$ possibilities for $e_{n}(\alpha)=k$. 

The formula for the sum is shown by induction. We have $W_{2}=1=1!+2!-2$, and for $n> 2$,
\begin{eqnarray}
\nonumber W_{n} &=& \sum_{k=1}^{n-1}k(n-k)(n-k)! = \sum_{k=0}^{n-2}(k+1)(n-1-k)(n-1-k)!\\
\nonumber          &=& \sum_{k=1}^{n-2}k(n-1-k)(n-1-k)! + \sum_{k=0}^{n-2}(n-1-k)(n-1-k)!\\
\nonumber          &=& W_{n-1} + \sum_{i=1}^{n-1}i.i! = 1! + \cdots + (n-1)!-(n-1)+n!-1=1! + \cdots +n!-n.
\end{eqnarray}
In the last line, we have used identity (\ref{factorial_system}) and the induction hypothesis.
\end{proof}

The $\varpi$-ruler sequence is analogous to the ruler sequence (sequence A001511 in Sloane \cite{oeis}). The difference is that the number of intermediate ticks increases with $n$ (Table \ref{table_E}).

\begin{table}[ht]
\footnotesize
\begin{center}
\begin{tabular}{| l | c |}
\hline
$n$ &$E_{n}$\\
\hline
2 &$1$\\
3 &$1^{2}21^{2}$\\
4 &$1^{3}21^{3}21^{3}31^{3}21^{3}21^{3}$\\
5 &$1^{4}21^{4}21^{4}21^{4}31^{4}21^{4}21^{4}21^{4}31^{4}21^{4}21^{4}21^{4}41^{4}21^{4}21^{4}21^{4}31^{4}21^{4}21^{4}21^{4}31^{4}21^{4}21^{4}21^{4}$\\
\hline
\end{tabular}
\end{center}
\caption{The $\varpi$-ruler sequence for $n = 2,3,4,5$ ($1^j$ denotes $1$ repeated $j$ times).}
\label{table_E}
\end{table}

\section{Combinatorial Gray code}\label{the_combinatorial_gray_code}
A combinatorial Gray code is a method for generating combinatorial objects so that successive objects differ by some pre-specified adjacency rule involving a minimality criterion (Savage \cite{savage_1997}). Such a code can be formulated as an Hamiltonian path or cycle in a graph whose vertices are the combinatorial objects to be generated. Two vertices are joined by an edge if they differ from each other in the pre-specified way.

The code associated with the $\varpi$-system corresponds to an Hamiltonian path in a weighted directed graph $G_{n}$.

\begin{definition}\label{weight}
The vertices of the digraph $G_{n}$ are the elements of $\mathcal{S}_{n}$. For two permutations (vertices) $p_{\alpha}$ and $p_{\alpha '}$, there is an arc from $p_{\alpha}$ to $p_{\alpha '}$ if and only if the last symbols of $p_{\alpha}$ match the first symbols of $p_{\alpha '}$ (there is no arc when there is no match). Let $p_{\alpha},p_{\alpha '} \in \mathcal{S}_{n}$ be two connected vertices in $G_{n}$. The \textit{weight} $f_{n}(\alpha,\alpha ') \in \{1, \ldots, n-1 \}$ associated with the arc $(p_{\alpha},p_{\alpha '})$ is the number of symbols that have to be erased to the left of $p_{\alpha}$ so that the last symbols of $p_{\alpha}$ match the first symbols of $p_{\alpha '}$.
\end{definition}

By Proposition \ref{AB_BA}, for each $\alpha$, there is an arc of weight $e_{n}(\alpha) = f_{n}(\alpha,\alpha+1)$ joining $p_{\alpha}$ to $p_{\alpha+1}$. This allows to define the Hamiltonian path
\begin{displaymath}
\mathbf{w}_{n} = \{(p_{\alpha},p_{\alpha+1}); \; \alpha = 0, \ldots, n!-2\}
\end{displaymath}
joining successive permutations. This path has total weight $W_{n}=1!+ \ldots + n!-n$ by Proposition \ref{ruler_sequence}. The path $\mathbf{w}_{n}$ can be closed into an Hamiltonian cycle by joining the last permutation $p_{n!-1}$ to the first $p_{0}$ by an arc of weight $n-1$:
\begin{displaymath}
(x_{n} \cdots x_{2}x_{1})  \xrightarrow{n-1} (x_{1}x_{2} \cdots x_{n}).
\end{displaymath}
Hence, an oriented path exists from any vertex to any other, so that $G_{n}$ is strongly connected.

Table \ref{table_G4} displays the weighted adjacency matrix of the digraph $G_{4}$ and the Hamiltonian path $\mathbf{w}_{4}$.

\begin{table}[ht]
\tiny
\begin{center}
\begin{tabular}{ c | c c c c |c c c c | c c c c || c c c c | c c c c | c c c c c|}
   &0 &1 &2 &3 &4 &5 &6 &7 &8 &9 &10 &11 &12 &13 &14 &15 &16 &17 &18 &19 &20 &21 &22 &23\\
\hline
0  &0 &\bf{1} &2 &3 &0 &0 &0 &3 &0 &0 &0 &3 &0 &0 &2 &3 &0 &0 &0 &3 &0 &0 &0 &3\\ 
1  &3 &0 &\bf{1} &2 &0 &0 &3 &0 &0 &3 &0 &0 &0 &3 &0 &0 &3 &0 &0 &2 &0 &0 &3 &0\\ 
2  &2 &3 &0 &\bf{1} &3 &0 &0 &0 &0 &2 &3 &0 &3 &0 &0 &0 &0 &0 &3 &0 &0 &3 &0 &0\\ 
3  &1 &2 &3 &0 &\bf{2} &3 &0 &0 &3 &0 &0 &0 &0 &0 &3 &0 &0 &3 &0 &0 &3 &0 &0 &0\\ 
\hline
4  &0 &0 &0 &3 &0 &\bf{1} &2 &3 &0 &0 &0 &3 &0 &0 &0 &3 &0 &0 &0 &3 &0 &0 &2 &3\\
5  &0 &3 &0 &0 &3 &0 &\bf{1} &2 &0 &0 &3 &0 &3 &0 &0 &2 &0 &0 &3 &0 &0 &3 &0 &0\\
6  &0 &2 &3 &0 &2 &3 &0 &\bf{1} &3 &0 &0 &0 &0 &0 &3 &0 &0 &3 &0 &0 &3 &0 &0 &0\\ 
7  &3 &0 &0 &0 &1 &2 &3 &0 &\bf{2} &3 &0 &0 &0 &3 &0 &0 &3 &0 &0 &0 &0 &0 &3 &0\\
\hline 
8  &0 &0 &0 &3 &0 &0 &0 &3 &0 &\bf{1} &2 &3 &0 &0 &0 &3 &0 &0 &2 &3 &0 &0 &0 &3\\ 
9  &0 &0 &3 &0 &0 &3 &0 &0 &3 &0 &\bf{1} &2 &0 &0 &3 &0 &0 &3 &0 &0 &3 &0 &0 &2\\ 
10 &3 &0 &0 &0 &0 &2 &3 &0 &2 &3 &0 &\bf{1} &0 &3 &0 &0 &3 &0 &0 &0 &0 &0 &3 &0\\
11 &2 &3 &0 &0 &3 &0 &0 &0 &1 &2 &3 &0 &\bf{3} &0 &0 &0 &0 &0 &3 &0 &0 &3 &0 &0\\
\hline
\hline
12 &0 &0 &2 &3 &0 &0 &0 &3 &0 &0 &0 &3 &0 &\bf{1} &2 &3 &0 &0 &0 &3 &0 &0 &0 &3\\ 
13 &0 &3 &0 &0 &3 &0 &0 &2 &0 &0 &3 &0 &3 &0 &\bf{1} &2 &0 &0 &3 &0 &0 &3 &0 &0\\ 
14 &3 &0 &0 &0 &0 &0 &3 &0 &0 &3 &0 &0 &2 &3 &0 &\bf{1} &3 &0 &0 &0 &0 &2 &3 &0\\
15 &0 &0 &3 &0 &0 &3 &0 &0 &3 &0 &0 &0 &1 &2 &3 &0 &\bf{2} &3 &0 &0 &3 &0 &0 &0\\
\hline
16 &0 &0 &0 &3 &0 &0 &0 &3 &0 &0 &2 &3 &0 &0 &0 &3 &0 &\bf{1} &2 &3 &0 &0 &0 &3\\ 
17 &3 &0 &0 &2 &0 &0 &3 &0 &0 &3 &0 &0 &0 &3 &0 &0 &3 &0 &\bf{1} &2 &0 &0 &3 &0\\ 
18 &0 &0 &3 &0 &0 &3 &0 &0 &3 &0 &0 &0 &0 &2 &3 &0 &2 &3 &0 &\bf{1} &3 &0 &0 &0\\ 
19 &0 &3 &0 &0 &3 &0 &0 &0 &0 &0 &3 &0 &3 &0 &0 &0 &1 &2 &3 &0 &\bf{2} &3 &0 &0\\
\hline
20 &0 &0 &0 &3 &0 &0 &2 &3 &0 &0 &0 &3 &0 &0 &0 &3 &0 &0 &0 &3 &0 &\bf{1} &2 &3\\ 
21 &0 &0 &3 &0 &0 &3 &0 &0 &3 &0 &0 &2 &0 &0 &3 &0 &0 &3 &0 &0 &3 &0 &\bf{1} &2\\
22 &0 &3 &0 &0 &3 &0 &0 &0 &0 &0 &3 &0 &3 &0 &0 &0 &0 &2 &3 &0 &2 &3 &0 &\bf{1}\\ 
23 &3 &0 &0 &0 &0 &0 &3 &0 &0 &3 &0 &0 &2 &3 &0 &0 &3 &0 &0 &0 &1 &2 &3 &0\\ 
\hline
\end{tabular}
\end{center}
\caption{The adjacency matrix of the digraph $G_{4}$. Lines delineate the 1-orbits. Double lines delineate the 2-orbits. Bold entries on the upper diagonal indicate the Hamiltonian path corresponding to the $\varpi$-system code, and forming the $\varpi$-ruler sequence.}
\label{table_G4}
\end{table}

\begin{proposition}\label{vertex_degree}
Each vertex of $G_{n}$ has exactly $j!$ in-arcs of weight $j$ and $j!$ out-arcs of weight $j$, $j=1, \ldots, n-1$. Hence the vertices of $G_{n}$ have $L=1!+ \cdots + (n-1)!$ as in- and out-degree, and $G_{n}$ is $L$-regular.
\end{proposition}
\begin{proof}
Let us consider the arcs of weight $j \in \{1, \ldots ,n-1\}$ joining a vertex to another in $G_{n}$:
\begin{displaymath}
(a_{1} \cdots a_{j}b_{1} \cdots b_{n-j}) \xrightarrow{j} (b_{1} \cdots b_{n-j}c_{1} \cdots c_{j}),
\end{displaymath}
where the c's are a permutation of the a's. There are $j!$ possibilities for the $c$'s, the $a$'s and the $b$'s being fixed. Hence $j!$ arcs of weight $j$ leave each vertex. Similarly, there are $j!$ possibilities for the $a$'s, the $b$'s and the $c$'s being fixed, so that $j!$ arcs of weight $j$ enter each vertex.
\end{proof}

We conjecture that the Hamiltonian path $\mathbf{w}_{n}$ joining successive permutations in the digraph $G_{n}$ is of minimal total weight. Assuming this conjecture we may state:\\

\textit{The $\varpi$-system is a combinatorial Gray code listing the permutations generated by cyclic shift. The adjacency rule is that the minimal number of symbols is erased to the left of a permutation so that the last symbols of the permutation match the first symbols of the next permutation.}

\section{Acknowledgements}
Marco Castera initiated the problem which motivated this study. Philippe Paclet discovered the weighted directed graph described in section \ref{the_combinatorial_gray_code}.

\bigskip
\hrule
\bigskip

\noindent 2000 \textit{Mathematics Subject Classification}: Primary 05A05.

\noindent \textit{Keywords}: permutations, cyclic shift, number system, palindrome, combinatorial Gray code.

\bigskip
\hrule
\bigskip

\end{document}